\documentclass{article}
\usepackage[latin1]{inputenc}
\usepackage{amsfonts}
\usepackage{amsmath, latexsym}
\usepackage{amssymb,verbatim}
\usepackage{theorem}
\usepackage{color}
\usepackage{euscript}

\newtheorem{thm}     {Theorem}
\newtheorem{lem} [thm] {Lemma}
\newtheorem{pro}  [thm]   {Proposition}

\theorembodyfont{\rmfamily}
\newtheorem{rem}[thm] {Remark}

\newcommand{\ba}{\begin{align*}}
\newcommand{\ea}{\end{align*}}

\newcommand{\R}{\mathbb{R}}
\newcommand{\Ss}{\mathbb{S}}

\newcommand{\eps}{\varepsilon}

\newcommand{\ds}{\displaystyle}

\newcommand{\qed}{\hfill$\square$}

\newcommand{\be}{\begin{equation}}

\newcommand{\ee}{\end{equation}}

\newcommand{\nd}{\noindent}

\newcommand{\NN}{\mathbb{N}}

\newcommand{\f}{\varphi}

\title{Uniqueness result for a weighted pendulum equation modeling domain walls in notched ferromagnetic nanowires
}

\author{ {\Large Radu Ignat}
\thanks{Institut de Math\'ematiques de Toulouse \& Institut Universitaire de France, UMR 5219, Universit\'e de Toulouse, CNRS, UPS
IMT, F-31062 Toulouse Cedex 9, France. Email: Radu.Ignat@math.univ-toulouse.fr}}

\begin{document}

\maketitle

\begin{abstract}
We prove an existence and uniqueness result for solutions $\f$ to a weighted pendulum equation in $\R$ where the weight is non-smooth and coercive. We also establish (in)stability results for $\f$ according to the
monotonicity of the weight. These results are applied in a reduced model for thin ferromagnetic nanowires with notches to obtain existence, uniqueness and stability of domain walls connecting two opposite directions of the magnetization.

\smallskip

\noindent {\it Keywords:} pendulum equation, ferromagnetism, domain walls, uniqueness, stability.
\end{abstract}

\section{Introduction}
We consider a weight $a:\R\to \R$ that is a bounded positive measurable function (not necessarily continuous) satisfying 
\be
\label{cond}
A_0\geq a(x)\geq a_0>0 \quad \textrm{ for a.e. } x\in \R.
\ee
Motivated by a reduced model for notched ferromagnetic thin nanowires (see \cite{CS}) where $a$ represents the area of transversal sections in the nanowire (the area $a$ may have jumps in that model), we associate to the magnetization  $m=(m_1, m_2, m_3):\R\to \Ss^2$ the following energy functional
$$F(m)=\int_{\R} \bigg(a(x)|\partial_x m|^2+a(x)(m_2^2+m_3^2)\bigg)\, dx.$$
We are interested in the analysis of domain walls that are transition layers connecting the opposite directions $\pm e_1$, where $e_1=(1,0,0)$. Up to a rotation and a translation (eventually yielding a translated weight), we fix the center of the domain wall at the origin by imposing $m(0)=e_2=(0,1,0)$. 
Our first theorem is the following uniqueness result for optimal domain walls:

\begin{thm}
\label{thm1}
There exists a unique minimizer $m:\R\to \Ss^2$ of $F$ under the constraints 
$$m(0)=e_2\quad \textrm{ and }\quad m(\pm\infty)=\pm e_1.$$ This minimizer has the form $m=(\sin \f, \cos \f, 0)$ where $\f:\R\to \R$ is an increasing Lipschitz function with 
\be
\label{constr}
\f(0)=0  \quad \textrm{and}\quad \f(\pm \infty)=\pm \frac\pi2
\ee
and $\f$ solves the weighted pendulum equation
\be
\label{EL}
\partial_x(a(x)\partial_x \f)+a(x)\sin \f \cos \f=0\quad \textrm{in}\quad \R\setminus\{0\}. 
\ee
If in addition, $a$ is even in $\R$ and non-decreasing in $\R_+=(0, +\infty)$, then $\f$ is odd in $\R$, \eqref{EL} holds in the entire $\R$ and $m$ is a stable critical point of $F$, i.e., for every
$v\in H^1(\R, \R^3)$ with $v\cdot m=0$ in $\R$,
$$
T(v)=\frac12\frac{d^2}{dt^2}\bigg|_{t=0} F\bigg(\frac{m+tv}{|m+tv|}\bigg)=\int_{\R}\bigg( a(x)|\partial_x v|^2+a(x)(v_2^2+v_3^2)-\lambda(x)|v|^2\bigg)\, dx\geq 0,
$$
where $\lambda(x)=a(x)|\partial_x m|^2+a(x)m_2^2$ is the Lagrange multiplier for the constraint $|m|=1$.
\end{thm}

Theorem \ref{thm1} is based on the following uniqueness result for solutions of the weighted pendulum equation \eqref{EL}.

\begin{thm}\label{lem1} There exists a unique solution $\f\in \dot{H}^1(\R)$ to \eqref{EL} in $\R\setminus \{0\}$ under the constraints \eqref{constr} that satisfies $-\frac\pi 2\leq \f(-x)\leq 0\leq \f(x) \leq \frac\pi 2$ for every $x>0$.
This solution $\f$ is Lipschitz and increasing in $\R$, $a(x)\partial_x \f\in W^{1, \infty}(\R\setminus \{0\})$ is positive and $\f$ is the unique minimizer of
\be
\label{star}
\min \bigg\{G(\f)=F\big((\sin \f, \cos \f,0)\big)\, :\, \f(0)=0  \quad \textrm{and}\quad \f(\pm \infty)=\pm \frac\pi2\bigg\}.\ee
Furthermore, if $a$ is even in $\R$ and non-decreasing in $\R_+$, then $\f$ is an odd stable solution to \eqref{EL} in the entire $\R$, i.e., for every $\eta\in H^1(\R)$,
$$Q(\eta)=\frac12\frac{d^2}{dt^2}\bigg|_{t=0}G(\f+t\eta)=\int_{\R} \bigg(a(x)(\partial_x \eta)^2-a(x)\cos 2\f(x) \, \eta^2\bigg)\, dx\geq 0.$$
\end{thm}

\begin{rem}
In Theorems \ref{thm1} and \ref{lem1}, the even symmetry of the weight $a$ is imposed to have the odd symmetry of the solution $\f$ yielding the equation \eqref{EL} to hold in the entire $\R$. Moreover, the monotonicity of $a$ is imposed to have the stability of the solutions $\f$ and $m$.
\end{rem}
 
Without the assumption that $a$ is non-decreasing in $\R_+$, the solution $\f$ in Theorem \ref{lem1}  can be unstable, i.e., $Q(\eta)<0$ in some direction $\eta\in H^1(\R)$ (yielding also the instability of the constraint minimizer $m$ in Theorem \ref{thm1}). We give the following example for a non-increasing weight $a$ in $\R_+$:

\begin{pro}
\label{pro1}
Let $a:\R\to \R$ be the even function given by $a=2$ in $(-1,1)$ and $a=1$ in $\R\setminus [-1,1]$. Then the solution $\f$ in Theorem \ref{lem1}  is unstable, i.e., $Q(\eta)<0$ in some direction $\eta\in H^1(\R)$. Consequently, the constraint minimizer $m$ in Theorem 1 is unstable, i.e., $T(v)<0$ for some direction $v\in H^1(\R, \R^3)$ with $v\cdot m=0$ in $\R$. Moreover, there is no minimizer $\tilde m:\R\to \Ss^2$ of $F$ under the constraints $\tilde m(\pm \infty)=\pm e_1$.
\end{pro}

\medskip

In the case of an even weight $a$ that is $C^1$ smooth in $\R$ and non-decreasing in $\R_+$, existence and stability results for domain walls are proved by Carbou and Sanchez in \cite{CS}. They address the uniqueness of domain walls as an open question. Theorems \ref{thm1} and \ref{lem1} give positive results for the question of uniqueness. The proof is based on a variational method for non-smooth and non-monotonous weight $a$ (instead of the shooting method used in  \cite{CS} where the regularity of $a$ is essential). The difficulty here consists in the heterogeneity of the non-smooth weight $a$ for which the equipartition of the two terms in the energy is lost in general for optimal domain walls (in contrast to the case of homogeneous weight yielding an autonomous ODE in \eqref{EL}).

\section{The weighted pendulum equation}

Note that the solutions $\f:\R\to \R$ to the weighted pendulum equation \eqref{EL} in $\R$ are critical points of the energy functional
$$G(\f)=F\big((\sin \f, \cos \f,0)\big)=\int_{\R} \bigg(a(x)(\partial_x \f)^2+a(x)\cos^2 \f\bigg)\, dx.$$
We start with the following existence result:

\begin{lem} \label{lem-exis} There exists a minimizer $\f:\R\to \R$ in \eqref{star}. Any such minimizer $\f$ satisfies 
$-\frac\pi 2\leq \f(-x)\leq 0\leq \f(x) \leq \frac\pi 2$ for every $x>0$ and \eqref{EL} holds in 
$\R\setminus\{0\}$. 
\end{lem}

\begin{proof}{ of Lemma \ref{lem-exis}} We divide the proof in several steps:

\smallskip

\nd {\it Step 1}. {\it Existence of a minimizer}. Let $(\f_n:\R \to \R)_n$ be a minimizing sequence in \eqref{star}. By cutting at $\pm \frac\pi2$, the energy density cannot increase, so we can assume $\f_n\in [-\frac\pi2, \frac\pi2]$ in $\R$. Moreover, replacing $\f_n$ by $|\f_n|$  in $\R_+$ and $(-|\f_n|)$ in $\R_-$, the energy density does not change, so we can assume $\f_n\in [0, \frac\pi2]$ in $\R_+$ and $\f_n\in [-\frac\pi2,0]$ in $\R_-$. As $(G(\f_n))_n$ is bounded, by \eqref{cond} we get that $(\f_n)_n$ is bounded in $\dot H^1(\R)$. Thus, for a subsequence, $\f_n\rightharpoonup \f$ weakly in $\dot H^1(\R)$ and uniformly in every compact of $\R$. In particular, $\f(0)=0$ and $\f_n\to \f$ pointwise in $\R$, so $\f \in [0, \frac\pi2]$ in $\R_+$ and $\f \in [-\frac\pi2,0]$ in $\R_-$. By l.s.c. of the map $\psi\mapsto \int_{\R} a(x) \psi^2\, dx$ in weak $L^2(\R)$ and Fatou's lemma, we deduce that $G$ is l.s.c. in weak $\dot H^1(\R)$, so that $\liminf_{n\to \infty} G(\f_n) \geq G(\f)$.  In particular, $\cos \f\in H^1(\R_\pm)$ yielding $\cos \f(\pm \infty)=0$. As $\f(+\infty)\in [0, \frac\pi2]$ and $\f(-\infty) \in [-\frac\pi2,0]$, we conclude that $\f(\pm\infty)=\pm \frac\pi2$, i.e., $\f$ is a minimizer in \eqref{star}. 

\smallskip

\nd {\it Step 2}. {\it Properties of any minimizer in \eqref{star}}. Let $\f$ be an arbitrary minimizer in \eqref{star}. Then $\f\in \dot H^1(\R)$ is continuous in $\R$. By minimality, $\f$ verifies \eqref{EL} in ${\cal D}'(\R\setminus\{0\})$. Thus, by \eqref{cond},  $\partial_x(a(x)\partial_x \f)\in L^\infty(\R\setminus\{0\})$ yielding $a(x)\partial_x \f$ is Lipschitz in every set $(-R,R)\setminus\{0\}$; in particular, $\partial_x \f$ is a bounded function in $(-R,R)$, so $\f$ is Lipschitz in $(-R,R)$ for every $R>0$. 

\smallskip

\nd {\it Claim 1}: {\it We prove $-\frac\pi2\leq \f\leq \frac\pi 2$ in $\R$}. For that, assume by contradiction that there is a point in $\R$ where $\f$ is larger than $\frac\pi 2$. By continuity of $\f$ and \eqref{constr}, it means that there is a non-empty interval $J=(x_0, y_0)$ such that $\f>\frac\pi 2$ in $J$ and $\f(x_0)=\f(y_0)=\frac\pi 2$. If we cut-off at $\frac\pi2$ and set $\tilde \f:=\f$ in $\R\setminus J$ and $\tilde \f=\frac\pi 2$ in $J$, then $\tilde \f$ satisfies the constraints \eqref{constr} and $G(\f)>G(\tilde \f)$ (as the energy of $\f$ in $J$ is positive while the energy of $\tilde \f$ in $J$ vanishes)  which contradicts the minimality of $\f$. Thus, $\f\leq \frac\pi2$ in $\R$. A similar argument shows that $\f\geq -\frac\pi2$ in $\R$.

\smallskip

\nd {\it Claim 2}: {\it We prove that $0$ is the only vanishing point of $\f$ in $\R$}. For that, assume by contradiction that there is a point $x_0\neq 0$ such that $\f(x_0)=0$. W.l.o.g., we may assume that $x_0>0$. Set $J=(0, x_0)$ and $\tilde \f=\f$ in $\R\setminus J$ and $\tilde \f=-\f$ in $J$. Then $\tilde \f$ is also a minimizer in \eqref{star} because $G(\f)=G(\tilde \f)$. By the regularity above, it means that $a(x)\partial_x \f$ and $a(x)\partial_x \tilde \f$ are continuous around $x_0\neq 0$. As $a(x)\partial_x \f(x)=a(x)\partial_x \tilde \f(x)$ if $x>x_0$ and $a(x)\partial_x \f(x)=-a(x)\partial_x \tilde \f(x)$ in $J$, we conclude that $a(x_0)\partial_x \f(x_0)=0$. Then we apply the unique continuation principle for the solution $\f_*=0$ to \eqref{EL} in $\R_+$: 
writing $\xi=a(x) \partial_x \f$, the ODE \eqref{EL} for $\f$ in $\R_+$ is equivalent to the linear ODE system $\partial_x(\f, \xi)=(\frac{\xi}{a(x)}, -a(x) b(x) \f)$ in $\R_+$ with bounded coefficients, where $b$ is given by 
$$b(x)=\begin{cases}\frac{\sin \f(x) \cos \f(x)}{\f(x)} & \textrm{ if  } \f(x)\neq 0,\\ 
1 & \textrm{ if  } \f(x)=0. \end{cases}$$
As $(\f(x_0), \xi(x_0))=(0,0)$, we have $\ds |(\f(x), \xi(x))|\leq C\bigg|\int_{x_0}^x |(\f(t), \xi(t))|\, dt\bigg|$ for every $x>0$ and Gronwall's lemma implies $(\f, \xi)=(0,0)$ in $\R_+$, in particular,
$\f=0$ in $\R_+$ which contradicts $\f(+\infty)=+\frac\pi 2$. Therefore, $\f$ vanishes only at $0$ in $\R$ which proves the claim. 

\smallskip

By continuity of $\f$ and \eqref{constr}, we conclude that $\f\in [-\frac\pi2,0]$ in $\R_-$ and $\f \in [0, \frac\pi2]$ in $\R_+$. 
\qed
\end{proof}

\bigskip

We prove now the uniqueness result for the weighted pendulum equation:

\begin{proof}{ of Theorem \ref{lem1}}
\nd Let $\f:\R\to \R$ be a $\dot H^1(\R)$ solution of \eqref{EL} in  $\R\setminus \{0\}$, $\f(0)=0$, $\f(\pm \infty)=\pm \frac\pi2$ and $\f\in [-\frac\pi2,0]$ in $\R_-$ and $\f \in [0, \frac\pi2]$ in $\R_+$. Then $\f$ is continuous in $\R$,  $\partial_x(a(x)\partial_x \f)\in L^\infty(\R\setminus\{0\})$ yielding $a(x)\partial_x \f$ is Lipschitz in $(-R,R)\setminus\{0\}$ and $\f$ is Lipschitz in $(-R,R)$ for every $R>0$. 

\smallskip

\nd {\it Step 1}. {\it We show that $\partial_x \f$ is a non-negative bounded function in $\R$.} We prove it in $\R_+$ (a similar argument yields also the conclusion in $\R_-$). As $(\partial_x \f)^2\in L^1(\R_+)$, we choose $x_n\to +\infty$ Lebesgue points of $\partial_x \f$ such that $\partial_x \f(x_n)\to 0$ as $n\to \infty$. As $0\leq \f\leq \frac\pi2$ in $\R_+$, we have by \eqref{EL} that the continuous function $a(x)\partial_x \f$ in $\R_+$ satisfies
\be
\label{234}
\partial_x(a(x)\partial_x \f)\leq 0\quad \textrm{ in}\quad  \R_+. 
\ee
Then for every $x>y>0$, we have for every large $n$ (so that $x_n>x$) that $a(y)\partial_x \f(y)\geq a(x)\partial_x \f(x)\geq a(x_n)\partial_x \f(x_n)\to 0$ as $n\to 0$. As $\partial_x \f$ is bounded around $0$, \eqref{cond} yields $+\infty>\frac{A_0}{a_0}\limsup_{y\searrow 0}\partial_x \f(y)\geq \partial_x \f(x)\geq 0$ for a.e. $x\in \R_+$, i.e., $\partial_x \f\in L^\infty(\R_+)$ and non-negative. 

\smallskip

\nd {\it Step 2}. {\it We show $\partial_x \f >0$ a.e. in $\R$.} We prove it first in $\R_+$. 
Assume by contradiction that there exists a Lebesgue point $x_0>0$ of $\partial_x \f$ such that  $\partial_x \f(x_0)=0$. By \eqref{234} and Step 1, as $a(\cdot)\partial_x \f$ is continuous in $\R_+$, we get $0=a(x_0)\partial_x \f(x_0)\geq a(x)\partial_x \f(x)\geq 0$ for every $x>x_0$, so $\partial_x \f=0$ a.e. in $(x_0, +\infty)$, that is, $\f=\frac\pi2$ in $(x_0, +\infty)$ by \eqref{constr}. Then the unique continuation principle\footnote{Writing $\psi=\frac\pi2-\f$ and $\xi=a(x) \partial_x \f$, the ODE \eqref{EL} in $\f$ in $\R_+$ is equivalent to the linear ODE system $\partial_x(\psi, \xi)=(-\frac{\xi}{a(x)}, -a(x) \tilde b(x) \psi)$ in $\R_+$ with bounded coefficients, where $\tilde b$ is given by 
$$\tilde b(x)=\begin{cases}\frac{\sin \f(x) \cos \f(x)}{\frac\pi 2-\f(x)} & \textrm{ if  } \f(x)\neq \frac\pi 2,\\ 
1 & \textrm{ if  } \f(x)=\frac\pi 2. \end{cases}$$
As $(\psi, \xi)=(0,0)$ in $(x_0, +\infty)$, Gronwall's lemma implies $(\psi, \xi)=(0,0)$ in $\R_+$, i.e., $\f=\frac\pi2$ in $\R_+$.
} 
for the solution $\f_*=\frac\pi2$ in \eqref{EL} implies  $\f=\f_*$ in $\R_+$ which contradicts $\f(0)=0$. Therefore, $\partial_x \f >0$ a.e. in $\R_+$. Also, by Step 1 (through \eqref{234}), we have $\liminf_{x\searrow 0}\partial_x \f>0$. A similar argument yields also the conclusion in $\R_-$ which finishes Step 2. 

\smallskip

\nd  By Step 2, $\f$ is Lipschitz and increasing in $\R$, in particular, $0$ is the only vanishing point of $\f$ and $-\frac\pi2<\f<\frac\pi 2$ in $\R$. As $a$ satisfies \eqref{cond}, by \eqref{EL} and $\f\in \dot{H}^1$, we deduce that $a(x)\partial_x \f\in W^{1, \infty}\cap L^2(\R\setminus \{0\})$ with $a(x)\partial_x \f >0$ in $\R\setminus \{0\}$.

\smallskip

\nd {\it Step 3}. {\it We prove that $G(\f)<\infty$}. As $\f\in \dot{H}^1$ is continuous in $\R$ and $a$ satisfies \eqref{cond}, it is enough to prove $\cos^2\f\in L^1(\R\setminus [-1,1])$. For that, multiplying \eqref{EL} by $\f$, integration by parts implies for every $y>1$:
\begin{align*}
&C\geq -\big[a(x)\partial_x \f(x) \f(x)\big]_{1}^{y}+\int_1^{y} a(x) (\partial_x \f)^2\,dx=-\int_1^{y}\partial_x (a(x)\partial_x \f) \f\, dx\\
&=\int_1^{y} a(x)\sin \f \cos \f \f \, dx\geq a_0 \f(1)\sin \f(1) \int_{1}^y \cos \f\, dx\geq \tilde C \int_{1}^y \cos^2 \f\, dx\end{align*}
where $C, \tilde C>0$ do not depend on $y$ because $\f$ and $a(x)\partial_x \f$ are Lipschitz in $(1, \infty)$. Passing to the limit $y\to \infty$, we get $\cos \f \in L^2((1, \infty))$. The same argument yields $\cos \f \in L^2((-\infty,1))$.

\smallskip

\nd {\it Step 4}. {\it We show that $\f$ is a minimizer in \eqref{star}}. For that, identifying $\R^2\sim \R^2\times \{0\}$, we set $m=(\sin \f, \cos \f, 0):\R\to  \Ss^1\times \{0\}\sim \Ss^1$. Then $m\in W^{1, \infty}\cap \dot H^1(\R, \Ss^1\times \{0\})$ and by \eqref{EL}, $m$ satisfies
\be
\label{EL-m}
-\partial_x(a(x)\partial_x m)+a(x)m_2 e_2=\lambda(x) m\quad \textrm{in}\quad \R\setminus \{0\}
\ee
with the Lagrange multiplier $$\lambda(x)=a(x)|\partial_x m|^2+a(x)m_2^2=a(x)(\partial_x \f)^2+a(x)\cos^2\f \in L^\infty(\R).$$ As $\f$ vanishes only at $0$ and $-\frac\pi2<\f<\frac\pi 2$ in $\R$, we have $|m_1|=|\sin \f|>0$ in $\R\setminus \{0\}$ and $m_2=\cos \f > 0$ in $\R$. Let now $\tilde \f$ be an arbitrary $\dot H^1(\R)$ function satisfying the constraints $\tilde \f(0)=0$ and $\tilde \f(+\infty)=\pm \frac\pi2$. We want to show that $G(\tilde \f)\geq G(\f)$. If $G(\tilde \f)=+\infty$, we are done. Otherwise, we may assume that $G(\tilde \f)<\infty$. We set 
$\tilde m=(\sin \tilde \f, \cos \tilde \f, 0):\R\to  \Ss^1\times \{0\}$ and write
$$\tilde m=m+v, \quad m\cdot v=-\frac{|v|^2}2, \quad v=(v_1,v_2,0)\in H^1(\R, \R^2\times \{0\}), \, v(0)=(0,0,0).$$
This is because $v_2=\tilde m_2-m_2\in L^2(\R)$; to show that $v_1=\tilde m_1-m_1\in L^2(\R_+)$, one uses that $\f(x), \tilde \f(x)\in [\pi/4, 3\pi/4]$ for all large $x$ and then $|\sin \alpha-\sin \beta|\leq C|\cos \alpha-\cos \beta|$ for every $\alpha, \beta \in [\pi/4, 3\pi/4]$, similarly $v_1\in L^2(\R_-)$. By \eqref{EL-m}, we compute using the second variation $T$ of $F$ at $m$:
\begin{align*}
G(\tilde \f)-G(\f)&=F(\tilde m)-F(m)\\
&=\int_{\R} a(x)|\partial_x v|^2+a(x)v_2^2+2a(x)\big(\partial_x m\cdot \partial_x v+ m_2 v_2)\, dx\\
&=\int_{\R} a(x)|\partial_x v|^2+a(x)v_2^2+2\lambda(x) m\cdot v\, dx\\
&=T(v)=(L_1 v_1,v_1)+(L_2 v_2,v_2)
\end{align*}
where $(\cdot, \cdot)$ is the duality $(H^{-1}, H^1)$ in $\R$ and $L_1$ and $L_2$ are the linear operators
\be
\label{L}
L_1=-\partial_x(a(x)\partial_x)-\lambda(x), \quad L_2=-\partial_x(a(x)\partial_x)+a(x)-\lambda(x).\ee
The minimality of $\f$ comes by the following:

\smallskip

\nd {\it Claim}: {\it We prove that $(L_1 v_1,v_1)\geq 0$ and $(L_2 v_2,v_2)\geq 0$ with equality if and only if $v_1=v_2=0$}. For that, we use the following method (called Hardy's decomposition, see e.g., \cite{INSZ, INSZ-ENS, IN}): if $\tilde v_1\in C^\infty_c(\R\setminus \{0\})$ and $\tilde v_2\in C^\infty_c(\R)$, then one decomposes $\tilde v_j=m_j \hat v_j$ (yielding $\hat v_j$ is Lipschitz with compact support in $\R$ since $m_j$ does not vanish on the support of $\tilde v_j$) and uses that $L_j m_j=0$ in $\R\setminus \{0\}$ for $j=1,2$, so that (e.g., see \cite[Lemma A.1]{INSZ} or Lemma \ref{thm:hardy} in Appendix):
$$(L_1 \tilde v_1, \tilde v_1)=\int_{\R} a(x) m_1^2 \bigg(\partial_x (\frac{\tilde v_1}{m_1})\bigg)^2\, dx, \quad (L_2 \tilde v_2, \tilde v_2)=\int_{\R} a(x) m_2^2 \bigg(\partial_x (\frac{\tilde v_2}{m_2})\bigg)^2\, dx.$$
To conclude, note that $\psi\mapsto (L_j \psi, \psi)$ is continuous in strong $H^1$ topology (as $a, \lambda\in L^\infty$) and $v_j$ can be approximated in strong $H^1(\R)$ by the above $\tilde v_j$, $j=1,2$. Therefore, Fatou's lemma yields\footnote{Note that $\frac{v_1}{m_1}$ may jump at $0$, therefore $(L_1 v_1,  v_1)$ is estimated by the integral taken only in $\R\setminus \{0\}$.}
\be
\label{positiv}
(L_1 v_1,  v_1)\geq \int_{\R\setminus \{0\}} a(x) m_1^2 \bigg(\partial_x (\frac{v_1}{m_1})\bigg)^2\, dx, \quad (L_2 v_2, v_2)\geq \int_{\R} a(x) m_2^2 \bigg(\partial_x (\frac{v_2}{m_2})\bigg)^2\, dx.
\ee
If $(L_1 v_1,  v_1)=(L_2 v_2,  v_2)=0$, we deduce that $v_1=c_1^\pm m_1$ in $\R_\pm$ and $v_2=c_2 m_2$ in $\R$ for some constants $c_1^\pm$ and $c_2$. As $v_1\to 0$ and $m_1\to \pm 1$ as $x\to \pm \infty$, we have $c_1^\pm=0$ yielding $v_1=0$ in $\R$. As $m_2(0)=1$ and $v_2(0)=0$, we also have $c_2=0$ yielding $v_2=0$ in $\R$; thus, $\tilde m=m$ in $\R$. 

\smallskip

\nd {\it Step 5}. {\it Uniqueness of $\f$}. Assume that $\tilde \f\in \dot H^1$ is another solution to \eqref{EL} in $\R\setminus \{0\}$ with the stated assumptions. Then by Step 4, $\tilde \f$ is a minimizer in \eqref{star}, so for $\tilde m=(\sin \tilde \f, \cos \tilde \f, 0)$, we get $G(\tilde \f)=F(\tilde m)=F(m)=G(\f)$ yielding $\tilde m=m$ in $\R$ and then $\f=\tilde \f$ in $\R$ by the uniqueness\footnote{\label{foot} Recall that $(\sin, \cos):[-\frac\pi2, \frac\pi2]\to \{(z_1,z_2)\in \Ss^1\, :\, z_2\geq 0\}$ is a diffeomorphism.} of the lifting $\f, \tilde \f\in [-\frac\pi2, \frac\pi2]$ in $\R$. In particular, $\f$ is the unique minimizer in \eqref{star}.

\smallskip

\nd {\it Step 5bis}. {\it Another proof for the uniqueness of $\f$}. This second method is based on a convexity argument inspired by \cite[Proposition 1]{IM}. For that, denoting $m_2=\cos \f$, we have
$$G(\f)=E(m_2)=\int_{\R}\bigg( a(x) \frac{(\partial_x m_2)^2}{1-m_2^2}+ a(x) m_2^2\bigg)\, dx.$$
Note that the function $(v,w)\in \R\times (-1,1)\mapsto \frac{v^2}{1-w^2}$ is convex and $w\in \R\mapsto w^2$ is strictly convex. Therefore, restricting ourselves to functions $m_2:\R\setminus \{0\}\to (-1,1)$, the functional $m_2\mapsto E(m_2)$ is strictly convex yielding the uniqueness of a critical point $m_2:\R\setminus \{0\}\to (-1,1)$ of $E$ satisfying $m_2(0)=1$ and $m_2(\pm \infty)=0$. This yields the uniqueness of critical points $\f$ of $G$ satisfying \eqref{constr}, $\f\in (-\pi,0)$ in $\R_-$ and $\f\in (0, \pi)$ in $\R_+$.

\smallskip

\nd {\it Step 6}. {\it If $a$ is even in $\R$ and non-decreasing in $\R_+$, then $Q(\eta)\geq 0$ for every $\eta\in H^1(\R)$.} This is true whenever $\eta(0)=0$ by the minimality of $\f$ in \eqref{star}: $G(\f+t\eta)\geq G(\f)$ for every $t\in \R$  (as $(\f+t\eta)(0)=0$ and $(\f+t\eta)(\pm \infty)=\pm \frac\pi 2$) yielding $Q(\eta)\geq 0$. Let us treat the general case (i.e., $\eta(0)\neq 0$). Note first that the symmetry of $a$ implies that the unique solution $\f$ is odd (as $\tilde \f(x)=-\f(-x)$ if $x<0$ and $\tilde \f=\f$ in $\R_+$ is also a solution of \eqref{EL} in $\R\setminus \{0\}$ satisfying our assumptions). Thus, the limits $\ds \lim_{x\searrow 0} a(x) \partial_x \f(x)=\lim_{x\nearrow 0} a(x)  \partial_x \f(x)$ exist and are equal 
(because $a$ and $\partial_x \f$ are even and $a(x)\partial_x \f$ is non-increasing in $\R_+$ by \eqref{234}), so \eqref{EL} holds in ${\cal D}'(\R)$ and 
\be
\label{2star}
\xi=a(x)\partial_x \f\in W^{1, \infty}\cap H^1(\R) \quad \textrm{and} \quad \xi>0 \, \textrm{ in } \, \R. 
\ee
As $a(x)\partial_x \f$ is non-increasing and $a$ is non-decreasing in $\R_+$, we deduce that $\partial_x \f$ is  non-decreasing in $\R_+$(so $\f$ is concave in $\R_+$), in particular, there exists the limit $\lim_{x\searrow 0} \partial_x \f(x)>0$. Next we compute
$$Q(\eta)=(L_0\eta, \eta), \quad L_0=-\partial_x(a(x)\partial_x)-a(x)\cos 2\f(x),$$
where $(\cdot, \cdot)$ is the duality $(H^{-1}, H^1)$. By \eqref{EL}, we have
$$L_0\xi=\sin \f \cos \f \partial_x(a^2) \quad \textrm{in } {\cal D}'(\R).$$ As $\f\in (0, \frac\pi2)$ and $a^2$ is non-decreasing in $\R_+$, by the symmetry of $\f$ and $a$, we deduce that $L_0\xi \geq 0$ in $H^{-1}(\R)$. We use Hardy's decomposition: by \eqref{2star}, if $\eta \in C^\infty_c(\R)$, then one decomposes $\eta=\xi \hat \eta$ and computes (e.g., see \cite[Lemma A.1]{INSZ}  or Lemma \ref{thm:hardy} in Appendix):
\be
\label{main-q}
Q(\eta)=(L_0\eta, \eta)=(L_0 \xi, \xi \hat \eta^2)+\int_{\R} a(x)\xi^2 (\partial_x\hat \eta)^2\, dx\geq \int_{\R} a(x)\xi^2 (\partial_x\hat \eta)^2\, dx\geq  0.\ee
As $a$ is bounded, $Q$ is continuous over $H^1(\R)$; therefore, by density of $C^\infty_c(\R)$ in $H^1(\R)$, 
Fatou's lemma yields $\ds Q(\eta)\geq  \int_{\R} a(x)\xi^2 \bigg(\partial_x (\frac{\eta}\xi)\bigg)^2\, dx\geq  0$ for every $\eta\in H^1(\R)$. Note that $Q(\eta)=0$ implies that $\eta=c\xi$ for some $c\in \R$. In general, the kernel of the quadratic form $Q$ (i.e., $\ker L_0$) may not be $\{0\}$: for example, if $a=1$ in $\R$, then $\ker L_0=\R \partial_x \f$ (which is due to the translation invariance of $G$ in $x$ for the homogeneous weight $a$).
\qed
\end{proof}

\section{Uniqueness of domain walls}

We use Theorem \ref{lem1} to prove the uniqueness of domain walls in Theorem \ref{thm1}.

\begin{proof}{ of Theorem \ref{thm1}} Let $\f$ be the unique minimizer in \eqref{star} given in Theorem \ref{lem1} and set $m=(\sin \f, \cos \f,0):\R\to \Ss^1\times \{0\}\sim \Ss^1$. Recall that $|m_1|=|\sin \f|>0$ in $\R\setminus \{0\}$ and  $m_2=\cos \f >0$ in $\R$.

\smallskip

\nd {\it Step 1. We prove that $m$ is the unique minimizer of $F$ under the constraints }
\be
\label{2con}
m(0)=e_2\quad \textrm{ and }\quad m(\pm \infty)=\pm e_1.
\ee
For that, let $\hat m:\R\to \Ss^2$ be an arbitrary map satisfying the constraints \eqref{2con}. We want to prove $F(\hat m)\geq F(m)$. W.l.o.g., we may assume that $F(\hat m)<\infty$. As $|\partial_x(\hat m_2, \hat m_3)|\geq |\partial_x \sqrt{\hat m_2^2+\hat m_3^2}|$ a.e. in $\R$, we deduce that 
$$
F(\hat m)\geq F(\tilde m), \quad \textrm{with} \quad \tilde m=(\hat m_1, \sqrt{\hat m_2^2+\hat m_3^2},0),
$$
where $\tilde m\in \dot H^1(\R, \Ss^1)$ is continuous satisfying the constraints \eqref{2con}. 
By the argument in Step 4 in the proof of Theorem \ref{lem1}, we know that $F(\tilde m)\geq F(m)$ with equality if and only if $\tilde m=m$. This proves the minimality of $m$ for $F$ under the constraints \eqref{2con}. If $\hat m:\R\to \Ss^2$ is another such minimizer, then (within the above notations) $F(\hat m)=F(\tilde m)=F(m)$. We deduce then that $\tilde m=m$ (in particular, 
$\hat m_1=m_1$ and $\sqrt{\hat m_2^2+\hat m_3^2}=m_2>0$ in $\R$) and $|\partial_x(\hat m_2, \hat m_3)|=|\partial_x \sqrt{\hat m_2^2+\hat m_3^2}|$
a.e. in  $\R$. This yields $\partial_x \frac{(\hat m_2, \hat m_3)}{\sqrt{\hat m_2^2+\hat m_3^2}}=(0,0)$ a.e. in $\R$. Together with the constraint
$\hat m(0)=e_2$, we conclude that $\hat m_3=0$ and $\hat m_2=m_2$ in $\R$, i.e., $\hat m=m$ in $\R$.

\bigskip

\smallskip

\nd {\it Step 2. If $a$ is even in $\R$ and non-decreasing in $\R_+=(0, +\infty)$, then $m$ is a stable critical point of $F$.} 
By Theorem \ref{lem1}, we know that $\f$ is odd in $\R$, \eqref{EL} holds in the entire $\R$ implying that \eqref{EL-m} holds in the entire $\R$ (so $m$ is a critical point of $F$ in $\R$) and $Q(\eta)\geq 0$ for every $\eta \in H^1(\R)$.
Let 
$v=(v_1, v_2, v_3) \in H^1(\R, \R^3)$ with $v\cdot m=0$ in $\R$ and denote $v'=(v_1, v_2,0)$. As $m(0)=e_2$, we have $v_2(0)=0$. The second variation $T$ of $F$ at $m$ is given by
$$
T(v)=\int_{\R}\big( a(x) |\partial_x v|^2+a(x)(v_2^2+v_3^2)-\lambda(x)|v|^2\big)\, dx=T(v')+(L_2 v_3,v_3),
$$
where $(\cdot, \cdot)$ is the duality $(H^{-1}, H^1)$ in $\R$ and $L_2$ is given in \eqref{L}. By Step 4 in the proof of Theorem~\ref{lem1}, we know that 
$(L_2 v_3,v_3)\geq 0$ for every $v_3\in H^1(\R)$. It remains to show that $T(v')\geq 0$ for every $v'\in H^1(\R, \R^2\times \{0\})$ with $v_1m_1+v_2m_2=0$ in $\R$ (in particular, $v_2(0)=0$). 

\smallskip

\nd {\it Case 1: $v_1$ is Lipschitz with compact support in $\R$ and $v_2$ is Lipschitz with compact support in  $\R\setminus \{0\}$.} In this case, the tangential constraint $v_1\sin \f+v_2\cos \f=0$ yields a Lipschitz function $\eta$ with compact support in $\R\setminus \{0\}$ (in particular, $\eta\in H^1(\R)$) such that
$\ds \eta=\frac{v_1}{\cos \f}=-\frac{v_2}{\sin \f}$ in $\R$. Then one checks $\eta^2=|v'|^2$, $|\partial_x v'|^2=(\partial_x \eta)^2+\eta^2 (\partial_x \f)^2$ and
$T(v')=Q(\eta)\geq 0$ as $\f$ is a stable critical point of $G$.

\smallskip

\nd {\it Case 2. The general case}. We can approximate $v'$ in strong $H^1(\R, \R^2\times \{0\})$ by vector fields $v'_n=(v_{1,n}, v_{2,n}, 0)$ such that $v_{1,n}\in C^\infty_c(\R)$ and $v_{2,n} \in C^\infty_c(\R\setminus \{0\})$. As $v'_n$ is not necessarily orthogonal to $m$ in every point, we consider the projection $\tilde v'_n=v'_n-
(m\cdot v'_n) m$ that also converges to $v'$ in $H^1(\R)$ (as $m, \partial_x \f\in L^\infty(\R)$) and satisfies the tangential constraint $\tilde v'_n\cdot m=0$ in $\R$. As $m$ is Lipschitz, Case 1 applies to $\tilde v_n'$ and  
$T(\tilde v'_n)\geq 0$. By the continuity of $T$ in $H^1(\R)$ (as $a$ and $\lambda$ are bounded in $\R$), we conclude that $T(v')\geq 0$.\qed 
\end{proof}

\section{Example of an unstable solution}

We choose the even weight $a=2$ in $(-1,1)$ and $a=1$ in $\R\setminus [-1,1]$ that clearly is non-increasing in $\R_+$. The aim is to prove that the solution $\f$ in Theorem \ref{lem1} (which is odd and satisfies \eqref{EL} in the entire $\R$) is unstable. An important feature for this weight is the nonexistence of minimizers in \eqref{star} if the constraint $\f(0)=0$ is dropped; this yields the nonexistence of domain walls connecting $\pm e_1$ if the center of the domain wall is not fixed.

\begin{proof}{ of Proposition \ref{pro1}} We divide the proof in several steps.

\smallskip

\nd {\it Step 1. Computation of the solution $\f$ in Theorem \ref{lem1}}. For that, as $\f$ is odd, it is enough to determine $\f$ in $\R_+$. The main observation is that
$(\partial_x \f)^2-\cos^2 \f$ is locally constant in $\R_+\setminus \{1\}$ (which follows by multiplying \eqref{EL} by $\partial_x \f$). These two constants are given by:
$$\begin{cases}
(\partial_x \f)^2-\cos^2 \f=d^2-1 & \textrm{in } (0,1),\\
(\partial_x \f)^2-\cos^2 \f=0 & \textrm{in } (1, \infty),
\end{cases}
\quad \textrm{with } \, d=\lim_{x\searrow 0}\partial_x \f(x)>0. 
$$
This is because of \eqref{constr} and the existence of Lebesque points $x_n\to \infty$ of $\partial_x \f$ such that $\partial_x \f(x_n)\to 0$.
As $\f\in \dot{H}^1(\R)$ is increasing with values in $[0, \frac\pi2]$ in $\R_+$, we deduce 
$$\partial_x \f=\sqrt{\cos^2 \f+d^2-1} \textrm{  in } (0,1), \quad \partial_x \f=\cos \f \textrm{  in } (1, \infty).$$
The aim is to determine $d$ (which is unique as $\f$ is unique by Theorem \ref{lem1}). For that, the continuity of $\xi=a(x)\partial_x \f$ in \eqref{2star} yields $\lim_{x\searrow 1}\partial_x \f(x)=2 \lim_{x\nearrow 1}\partial_x \f(x)$, that is, $\cos \f(1)=2 \sqrt{\cos^2 \f(1)+d^2-1}$. Thus, $d\in (\frac12,1)$ is given by $d=\sqrt{1-\frac34 \cos^2\f(1)}$. In other words, $d$ is the unique solution in $(\frac12,1)$ of the equation
$$1=\int_0^{\arccos \sqrt{4(1-d^2)/3}}\frac{dt}{\sqrt{\cos^2 t+d^2-1}}.$$

\smallskip

\nd {\it Step 2. We prove the instability of $\f$}. This is based on \eqref{main-q}. Indeed, let $\psi: \R_+\to \R_+$ be a smooth function such that $\psi=1$ in $(0,1)$ and $\psi=0$ for $x>2$. For every $\eps>0$, set $\hat \eta_\eps\in C^\infty_c(\R)$, $\hat \eta_\eps(x)=1$ if $|x|<1$ and $\hat \eta_\eps(x)=\psi\big(\eps(|x|-1)\big)$ for $|x|>1$. As $\partial_x(a^2)=-3\delta_1$ in ${\cal D}'(\R_+)$ and $\xi$ is Lipschitz in $\R$, by the symmetry of our functions, we get in \eqref{main-q} for the Lipschitz function $\eta_\eps=\xi \hat \eta_\eps$ with compact support in $\R$:
\begin{align*}
\frac12Q(\eta_\eps)&=-3\sin \f(1)\cos^2\f(1)+\int_{1}^\infty\xi^2 (\partial_x \hat \eta_\eps)^2\, dx\\
&\leq -3\sin \f(1)\cos^2\f(1)+\eps^2 \|\partial_x \psi\|^2_{L^\infty} \int_{1}^\infty (\partial_x \f)^2\, dx\stackrel{\eps\to 0}{\to} -3\sin \f(1)\cos^2\f(1)<0.
\end{align*}
Therefore, for $\eps$ small enough, $\eta_\eps\in H^1(\R)$ satisfies  $Q(\eta_\eps)<0$. This entails the instability of $m=(\sin \f, \cos \f, 0)$: indeed, setting $v=(\eta_\eps \cos \f, -\eta_\eps \sin \f, 0)\in H^1(\R, \R^3)$, then $v\cdot m=0$ in $\R$ and $T(v)=Q(\eta_\eps)<0$.

\smallskip

\nd {\it Step 3. Non-existence of minimizers in \eqref{star} in the absence of the constraint $\f(0)=0$}. The aim is to show that
$$\inf \bigg\{G(\hat \f)\, :\, \hat \f(\pm \infty)=\pm \frac\pi2\bigg\}=4$$
and this infimum is not achieved. For that, as $a\geq 1$ in $\R$, the Cauchy-Schwartz inequality implies
\be
\label{rad}
G(\hat \f)\geq \int_{\R} (\partial_x \hat \f)^2+ \cos^2 \hat \f \, dx\geq 2\int_{\R} \cos \hat \f \, \partial_x \hat \f\, dx=4\ee
because $\sin \hat \f(\pm\infty)=\pm1$. The last inequality in \eqref{rad} becomes equality if and only if $\partial_x \hat \f=\cos \hat \f$ in $\R$ yielding the existence of a center $x_0\in \R$ such that $\hat \f(x)=\frac\pi 2-2\arctan (e^{-x+x_0})$ for every $x\in \R$ and $\hat \f(x_0)=0$ (see e.g. \cite{CL, CI}). For this family of domain walls with center at $x_0\in \R$, one computes:
$$G(\hat \f)=4+\int_{-1}^1 (\partial_x \hat \f)^2+ \cos^2 \hat \f \, dx=4+2\int_{-1}^1  \cos \hat \f \, \partial_x \hat \f\, dx=4+2[\sin \hat \f]_{-1}^1\to 4$$ as $x_0\to \infty$. This proves that the above infimum is indeed equal to $4$. If this infimum is achieved for some $\f$, then both inequalities in \eqref{rad} become equalities, so $\f$ is one of the above domain walls $\hat \f$ with center at $x_0$ and $G(\hat \f)=4$ contradicting the fact that $2[\sin \hat \f]_{-1}^1>0$.

This implies the non-existence of a minimizer $m=(m_1,m_2,m_3):\R\to \Ss^2$ of $F$ under the constraints $m(\pm\infty)=\pm e_1$. Indeed, for every $m=(m_1,m_2,m_3):\R\to \Ss^2$ with $m(\pm\infty)=\pm e_1$, $F(m)\geq F(\tilde m)$ with $\tilde m=(m_1, \sqrt{m_2^2+m_3^2},0):\R\to \Ss^2$ and $\tilde m(\pm \infty)=\pm e_1$. By Footnote \ref{foot}, there exists a unique lifting $\tilde \f:\R\to [-\frac\pi2, \frac\pi2]$ such that $\tilde m=(\sin\tilde \f, \cos \tilde \f,0)$ and $\tilde \f(\pm \infty)=\pm \frac\pi2$. Thus, $F(\tilde m)=G(\tilde \f)>4$ and this infimum $4$ is never achieved by the above argument. \qed
\end{proof}

\section{Some open questions}

In Theorem \ref{lem1}, we proved existence and uniqueness of the minimizer $\f$ in \eqref{star}, in particular, under the constraint of a fixed center at the origin. A natural question is whether $\f$ is a minimizer of $G$ 
under the only two constraints $\f(\pm \infty)=\pm \frac\pi 2$. The answer is negative for some weights $a$ as shown in Proposition \ref{pro1} where $\f$ is unstable and moreover, no minimizers of $G$ exist under the  constraints $\f(\pm \infty)=\pm \frac\pi 2$. However, the answer is positive for the homogeneous weight $a$ where $\f$ is the unique minimizer and also, the unique critical point (up to a translation of the center) of $G$ under the constraints $\f(\pm \infty)=\pm \frac\pi 2$  (see e.g. \cite{CI}).

\smallskip

\nd {\bf Open question 1}:  {\it Under which additional condition on the weight $a$ satisfying \eqref{cond}, is it true that the solution $\f$ in Theorem \ref{lem1} is a minimizer of $G$ under the constraints $\f(\pm \infty)=\pm \frac\pi 2$? In that case, under which further conditions on $a$, $\f$ is the unique minimizer (or more, the unique critical point) of $G$ under the only two constraints $\f(\pm \infty)=\pm \frac\pi 2$? }

\smallskip

This addresses in particular the question of existence of a minimizer $\f$ of $G$ under the two constraints $\f(\pm \infty)=\pm \frac\pi 2$. Such problem is solved in general by using the concentration-compactness lemma \`a la Lions. For the homogeneous weight $a$, we recall the following compactness result that handles the type of constraints $\f(\pm \infty)=\pm \frac\pi 2$, i.e., transitions between two different states:

\begin{lem}[Doering-Ignat-Otto \cite{DIO}]
\label{dio}
Let $\f_n:\R\to \R$ be such that $\liminf_{x \nearrow +\infty} \f_n(x)>0$ and  $\limsup_{x\searrow -\infty} \f_n(x)<0$ for every $n\in \NN$. If $\ds \limsup_{n\to \infty}\|\partial_x \f_n\|_{L^2(\R)}<\infty$, then for a subsequence, there exists a zero $z_n$ of $\f_n$ and a limit $\f\in \dot H^1(\R)$ such that $\f(0)=0$ and 
$$\f_n(\cdot+z_n)\to \f \quad \textrm{ locally uniformly in $\R$ and weakly in $\dot H^1(\R)$}$$
and $\liminf_{x\nearrow +\infty} \f(x)\geq 0$ and  $\limsup_{x\searrow -\infty} \f(x)\leq 0$. 
\end{lem}

\nd We address the following question concerning the compactness of uniformly bounded energy configurations in the case of heterogeneous weights $a$:

\smallskip

\nd {\bf Open question 2}:  {\it Under which additional condition on the weight $a$ satisfying \eqref{cond}, is the following true: if $\f_n:\R\to \R$ satisfies $\liminf_{x \nearrow +\infty} \f_n(x)>0$, $\limsup_{x\searrow -\infty} \f_n(x)<0$ for every $n\in \NN$ and $\ds \limsup_n\|\partial_x \f_n\|_{L^2(\R)}<\infty$, then for a subsequence, there exists a zero $z_n$ of $\f_n$ and a limit $\f\in \dot H^1(\R)$ such that $\f(z)=0$ for some $z\in \R$, $\liminf_{x\nearrow +\infty} \f(x)\geq 0$ and  $\limsup_{x\searrow -\infty} \f(x)\leq 0$,  
$\f_n(\cdot+z_n)\to \f(\cdot+z)$  locally uniformly in $\R$ and }
$$\liminf_{n\to \infty} \int_{\R}a(x)(\partial_x \f_n)^2(x)\, dx\geq  \int_{\R}a(x)(\partial_x \f)^2(x)\, dx\, ?$$

In Theorem \ref{lem1}, in order to have uniqueness of the solution $\f$, we imposed the condition $\f \in [-\frac \pi2, 0]$ in $\R_-$ and $\f\in [0, \frac\pi2]$ in $\R_+$. This condition is satisfied by any minimizer $\f$ in \eqref{star} by Lemma~\ref{lem-exis}. We address the following question:

\smallskip

\nd {\bf Open question 3}: {\it If $a$ satisfies \eqref{cond}, is it true that  any solution $\f\in \dot{H}^1(\R)$ to \eqref{EL} in $\R\setminus \{0\}$ under the constraints \eqref{constr} satisfies $-\frac\pi 2\leq \f(-x)\leq 0\leq \f(x) \leq \frac\pi 2$ for every $x>0$?}

\bigskip

\medskip

\noindent{\bf Acknowledgment.} The author thanks Gilles Carbou for pointing out the question of uniqueness of domain walls in notched ferromagnetic nanowires. The author is partially supported by the ANR project MOSICOF ANR-21-CE40-0004.

\section*{Appendix}

Inspired by \cite[Lemma A.1]{INSZ}, we prove the following identity for non-smooth weights $a$:

\begin{lem}[Hardy's decomposition]\label{thm:hardy}
Let $a:\R\to \R$ satisfy \eqref{cond}, $V\in L^1_{loc}(\R)$ and 
$$L=-\partial_x (a(x)\partial_x)+V(x).$$
If $\psi \in W^{1, \infty}_{loc}(\R)$ satisfies $\psi>0$ in $\R$, then for every $\eta \in C^\infty_c(\R)$, writing $\hat \eta:=\frac \eta \psi$, we have the following Hardy decomposition: 
$$(L\eta, \eta)=(L\psi, \psi \hat \eta^2)+\int_{\R} a(x) \psi^2  (\partial_x \hat \eta)^2 \, dx,$$
where $(\cdot, \cdot)$ is the duality $(H^{-1}, H^1)$ in $\R$.
\end{lem}

\begin{proof}{} Note first that $\hat \eta$ is Lipschitz with compact support in $\R$. 
Integrating by parts, we have:
\begin{align*}
(L \eta, \eta)&= (L(\psi \hat \eta), (\psi \hat \eta))=\int_\R \bigg(a(x) \big(\partial_x(\psi \hat \eta)\big)^2 +V(x) \psi^2 \hat \eta^2\bigg)\, dx\\
&=\int_\R \bigg(a(x) \hat \eta^2  (\partial_x \psi)^2 +  a(x)\psi^2  (\partial_x \hat \eta)^2 +\frac 1 2 a(x) \partial_x(\psi^2) \partial_x (\hat \eta^2) +V(x) \psi^2 \hat \eta^2\bigg)\, dx\\
&{=} \int_\R a(x) \psi^2 (\partial_x \hat \eta)^2 \, dx+\int_\R \bigg( a(x) \hat \eta^2 (\partial_x \psi)^2+V(x) \psi^2 \hat \eta^2\bigg)\, dx-\frac 1 2 \bigg(\partial_x \big(a(x) \partial_x (\psi^2)\big), \hat \eta^2\bigg) \\ 
&=\int_\R a(x) \psi^2 (\partial_x \hat \eta)^2 \, dx+(L\psi, \psi \hat \eta^2),
\end{align*}
which is the desired identity. \qed 
\end{proof}

\end{document}